\documentclass[conference]{ieeeconf}
\IEEEoverridecommandlockouts
\usepackage{amsmath,amssymb,amsfonts} 
\usepackage{commath}
\newtheorem{theorem}{Theorem}

\newtheorem{lemma}{Lemma}

\newtheorem{definition}{Definition}

\newtheorem{example}{Example}
\usepackage{algorithmic}
\usepackage{graphicx}
\usepackage{textcomp}
\usepackage{xcolor}
\usepackage{caption}
\usepackage{subcaption}
\usepackage{tikz}
\usepackage{cite}
\usepackage{comment}
\usetikzlibrary{tikzmark,patterns,shapes,arrows,arrows.meta,chains,matrix,positioning,scopes,calc,patterns,shapes.geometric,intersections,backgrounds,fit,decorations.text}
\tikzset{
block/.style = {draw, rectangle,
	minimum height=0.5cm,
	minimum width=1cm},
input/.style = {coordinate,node distance=1cm},
output/.style = {coordinate,node distance=4cm},
arrow/.style={draw, -latex,node distance=2cm},
pinstyle/.style = {pin edge={latex-, black,node distance=2cm}},
sum/.style = {draw, circle, node distance=1cm},
}
\usetikzlibrary{positioning}
\usetikzlibrary{calc}
\usetikzlibrary{shapes, arrows}

\DeclareMathOperator{\diag}{diag}

\begin{document}

\title{Convergence Analysis of Distributed Optimization: \\A Dissipativity Framework
\\
}

\author{A. Karakai, J. Eising, A. Martinelli, F. Dörfler%
\thanks{A. Karakai, A. Martinelli, and F. Dörfler are with the Automatic Control Laboratory, ETH Zürich, Switzerland, {\tt\footnotesize \{akarakai, andremar, doerfler\}@ethz.ch}. J. Eising is with ENTEG, University of Groningen, the Netherlands,  {\tt\footnotesize j.eising@rug.nl}. This work was supported by the SNF/FW Weave Project 200021E\_20397.}}

\maketitle

\begin{abstract}
    We develop a system-theoretic framework for the structured analysis of distributed optimization algorithms with decomposable cost functions. We model such algorithms as a network of interacting dynamical systems and derive tests for convergence based on incremental dissipativity and contraction theory. This approach yields a step-by-step analysis pipeline suitable for any network structure, with conditions expressed as linear matrix inequalities. In addition, a numerical comparison with traditional analysis methods is presented, in the context of distributed gradient descent.
\end{abstract}

\section{Introduction}

The use of system-theoretic tools for the analysis and design of numerical algorithms has received significant research attention in recent years. As argued in \cite{Dörfler_He_Belgioioso_Bolognani_Lygeros_Muehlebach_2024}, the increasing prevalence of such algorithms in engineering systems requires a bridge between the analysis tools used in both domains. A prominent place in this is taken by optimization algorithms in closed-loop control, such as in model predictive control \cite{Rawlings2009-md} or online feedback optimization \cite{Colombino_Dall’Anese_Bernstein_2020,Hauswirth_He_Bolognani_Hug_Dörfler_2024}.

In this paper, we take this system-theoretic view in order to analyze distributed, or interconnected, optimization algorithms. Distributed optimization problems arise in a wide range of scenarios, including multi-agent control \cite{Molzahn_Dorfler_Sandberg_Low_Chakrabarti_Baldick_Lavaei_2017}, sensor networks \cite{Rabbat_Nowak_2004}, and privacy-preserving machine learning \cite{Li_Sahu_Talwalkar_Smith_2020}. In all of these, agents aim to optimize a shared global objective, relying solely on local computation while communication is restricted to a given network. For an overview of distributed optimization, see \cite{Yang_Yi_Wu_Yuan_Wu_Meng_Hong_Wang_Lin_Johansson_2019} and the references therein. Indeed, system-theoretic tools have proven useful in analyzing distributed optimization schemes. For instance, \cite{Notarstefano_2024} discusses how singular perturbation theory can be used for the analysis and design of distributed algorithms.

One of the main tools from systems theory used to analyze optimization algorithms is dissipativity theory \cite{Willems_1972a, Willems_1972b}. In particular, \cite{Lessard_2022} sets out a dissipativity-based approach for the convergence analysis of single optimization algorithms. The core idea there is to decompose such an algorithm into a feedback loop of a linear time-invariant (LTI) system and a static nonlinearity, such as a gradient operator. Monotone operator theory allows us to derive appropriate incremental bounds, called \textit{sector bounds}, on the latter from common assumptions on the optimization problem. Convergence is then verified by proving that the LTI system is asymptotically stable under any feedback that satisfies the sector bound. That is, convergence analysis reduces to the classical \textit{problem of absolute stability} \cite{Lure_Postnikov}. A parallel approach instead employs integral quadratic constraints \cite{Scherer_Ebenbauer_2021} to analyze the closed loop and to synthesize algorithms. 

Dissipativity theory places the analysis of systems in the context of interaction with the environment. Building on that, in \cite{onlineopt}, it was shown useful in analyzing robustness and performance in interconnections of an optimization algorithm with a plant. More generally, by characterizing system behaviour in terms of (abstract) energy exchange with the environment, it provides a powerful framework for analysis and design in a networked setting \cite{Arcak_2016, van_der_Schaft_2017, Aboudonia_Martinelli_Lygeros_2021, Martinelli_QSR, Nakano_Aboudonia_Eising_Martinelli_Dörfler_Lygeros_2025}. 

This observation, combined with the successful application of the theory to individual optimization algorithms, motivates a system-theoretic approach to distributed optimization algorithms based on dissipativity theory.

To be precise, we propose a system-theoretic framework for the analysis of distributed optimization algorithms through the incremental form of dissipativity theory \cite{Sepulchre_Chaffey_Forni_2022}. We model a distributed algorithm as a set of \textit{local} optimization algorithms interconnected over a network. Our aim is to guarantee convergence by proving that the network interconnection is \textit{contractive}, that is, the distance between pairs of trajectories shrinks over time (see e.g. \cite{Bullo_2022}). 

\subsection*{Contribution}

Our proposed framework provides a structured and systematic pipeline for analysis. In addition, it can handle any network topology, as well as heterogeneous distributed algorithms.
This framework
\begin{enumerate}
    \item yields semidefinite programs to verify (exponential) contraction of arbitrary linear interconnections of optimization algorithms,
    \item naturally links to monotone operator theory by employing the incremental form of dissipativity, and 
    \item places distributed optimization algorithms into the broader context of interconnected dynamical systems, allowing us to repurpose system-theoretic tools to analyze algorithms.
\end{enumerate}

\subsection*{Notation}
We use $I_n\in\mathbb{R}^{n\times n}$ to denote the identity matrix in dimension $n$. Dimensions are not written when they are clear from the context. When discussing discrete-time dynamical systems or iterative algorithms, $x^+$ and $x$ are to be understood as $x(k+1)$ and $x(k)$, respectively. The Euclidean norm is written as $\norm{\cdot}_2$. For symmetric matrices, the symbols $\prec$  ($\preceq$) and $\succ$ ($\succeq$) denote inequality in the (semi-) definite sense.

\section{Distributed Optimization}\label{sec:distr-opt}

In this paper, we consider decomposable optimization problems of the form
\begin{equation}\label{eq:opt-problem-global}
    \min_{x\in\mathbb{R}^n}\quad \sum_{i=1}^N f_i(x),
\end{equation}
where each $f_i:\mathbb{R}^n\to\mathbb{R}$ is a \textit{local} cost function that belongs to an agent in a network. In distributed optimization, the goal is to solve this \textit{global} problem while  performing all computations related to $f_i$, such as evaluations of the gradient $\nabla f_i$, at agent $i$ and restricting communication to a given network. A common approach is to reformulate \eqref{eq:opt-problem-global} as
\begin{equation}\label{eq:opt-problem-distributed}
    \begin{split}
        &\min_{x_1, x_2,\dots, x_N \in\mathbb{R}^{n}}\quad \sum_{i=1}^N f_i(x_i)\\
        &\textrm{subject to}\quad  x_1=x_2=\dots=x_N.
    \end{split}
\end{equation}

After this, \eqref{eq:opt-problem-distributed} is typically solved by iterative algorithms that combine updates aimed at minimizing $f_i$ with ones that asymptotically enforce the consensus constraint $x_1=x_2=\dots=x_N$. That is, the problem is decoupled into local optimization and distributed consensus estimation. An elaborate discussion of this idea from a system-theoretic viewpoint is given in \cite{Notarstefano_2024}.
In this paper, we consider algorithms where updates are linear combinations of optimization and consensus directions, based on communication through a network.
Examples of such methods include distributed gradient descent \cite{Yuan_Ling_Yin_2016} and EXTRA \cite{Shi_Ling_Wu_Yin_2015}.

\subsection{System-Theoretic Modelling}\label{sec:network-system-formulation}

The class of optimization algorithms we consider is inspired by \cite{Lessard_2022,onlineopt, Scherer_Ebenbauer_2021} and is introduced next. In line with those works, we consider agents modeled as discrete-time dynamical systems with inputs $u_i$ and outputs $y_i$. Adopting the standard notation of $u=[u_1^\top \hspace{0.5em} u_2^\top \ldots u_N^\top]^\top$, and similar for other signals, we model the communication between agents as a linear relation $u=My$, where $M$ can be, for example, the Laplacian or adjacency matrix of the (directed or undirected) communication graph.

The autonomous dynamics of the agents, that is, for $u=0$, constitute optimization algorithms. To model that, we follow the ideas of \cite{Lessard_2022} and write the optimization dynamics as a feedback loop between an LTI system and a static nonlinearity. The latter is denoted by $\varphi_i$ and referred to as an \textit{oracle}. It can be, for example, a gradient or a proximal operator. It is shown in \cite{Lessard_2022} that a wide range of common optimization algorithms, such as ADMM or Nesterov's accelerated method, can be cast in this form, and a number of case studies are presented there.

In particular, we consider local dynamics of the form
\begin{subequations}\label{eq:sigma}
\begin{align}
    \Sigma_i&{\footnotesize:\begin{cases}
        x_i^+ = A_i x_i + B_i u_i + G_i w_i,\\
        y_i = C^{\textrm{con}}_i x_i + D^{\textrm{con}}_i u_i + H^{\textrm{con}}_i w_i,\\
        z_i = C^{\textrm{opt}}_i x_i + D^{\textrm{opt}}_i u_i + H^{\textrm{opt}}_i w_i,
    \end{cases}} \label{eq:sigma-tilde-LTI}\\
    & \quad\,\,\,\, {\footnotesize w_i = \varphi_i(z_i),} \label{eq:sigma-tilde-oracle}
\end{align} \end{subequations}
for $i=1,2,\dots, N$. We refer to $u_i$ and $y_i$ as the \textit{interconnection input} and \textit{output}, and to $w_i$ and $z_i$ as the \textit{oracle input} and \textit{output}.
Figure~\ref{fig:network-block-diagram} illustrates this structure. We assume that the network interconnection $M$ and the feedthrough terms are such that the interconnection is well-posed.

\begin{example}\label{ex:DGD}
Suppose we have a scalar optimization problem, that is, $n=1$. Then, distributed gradient descent (see, e.g., \cite{Yuan_Ling_Yin_2016}) is traditionally given by the iterations
\begin{equation}\label{eq:dgd}
    x_i^+ = x_i - \eta \nabla f_i(x_i) - \rho (L x)_i
\end{equation}
for $i=1,2,\dots, N$, where $\eta, \rho > 0$, $L$ is the Laplacian of the connected and undirected communication graph, and $(Lx)_i$ is component $i$ of $Lx\in\mathbb{R}^N$. We can write \eqref{eq:dgd} as a network of optimization algorithms in the form \eqref{eq:sigma} as
\begin{equation}\label{eq:dgd-lti}
    {\small\Sigma_i:\begin{cases}
        x_i^+ = x_i - \rho u_i - \eta w_i,\\
        y_i = x_i,\\
        z_i = x_i,
    \end{cases}}
\end{equation}
with feedback $w_i = \nabla f_i(z_i)$ and coupling $u = Ly$. It is straightforward to extend this to the $n>1$ case. \hfill $\square$
\end{example}

\begin{figure}
    \centering
    \vspace{2mm}
    \input{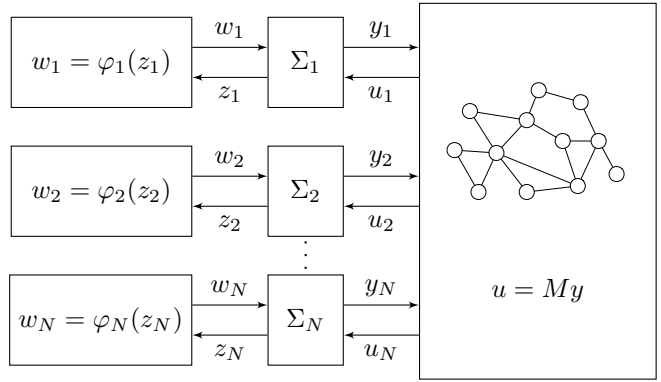}
    \caption{Block diagram of a distributed optimization algorithm modeled as a collection of LTI systems $\Sigma_i$ in feedback interconnection with their oracles $\varphi_i$, coupled via $u=My$. \vspace{-1.4em}}
    \label{fig:network-block-diagram}
\end{figure}

\subsection{Problem Formulation}

We are interested in proving the convergence of such distributed algorithms. In the language of dynamical systems, that translates to the existence of a fixed point and asymptotic stability of the system. Namely, the interconnected system should have a unique  fixed point that corresponds to the solution of the optimization problem. Satisfaction of the consensus constraint requires $y^*_1=y^*_2=\dots = y^*_N$, while optimality conditions depend on the choice of oracles. As an important example, for gradient-based methods with convex cost functions, $w^*_1 + w^*_2 + \dots + w^*_N = 0$
guarantees optimality at the fixed point. In the sequel, we assume that the algorithms at hand have the correct fixed points.

We view the problem of convergence through the lens of contraction theory, that is, by investigating the evolution of the distance between pairs of trajectories. To this end, we introduce the notation $\Delta x = x^1 - x^2$ for the difference between two trajectories and call $\Delta x$ an \textit{incremental trajectory}. We consider the following notion of contraction. 
\begin{definition}
    Let $\gamma \in (0,1)$. An autonomous discrete-time dynamical system $\Sigma$ with state $x\in\mathbb{R}^n$ is said to be nonexpansive, contractive, or exponentially contractive if there exists a positive definite quadratic function $V:\mathbb{R}^n \to \mathbb{R}$ such that, for any incremental trajectory $\Delta x \neq 0$, we have
    \begin{align*}
            V(\Delta x^+) - V(\Delta x) &\leq 0,\quad \textrm{or}\\ 
            V(\Delta x^+) - V(\Delta x) &< 0,\quad \textrm{or}\\ 
            V(\Delta x^+) - \gamma V(\Delta x) &\leq 0,
        \end{align*}
    respectively. The function $V$ is called a contraction metric, and $\gamma$ is called a contraction rate.
\end{definition}

For a comprehensive overview of contraction theory, including more general contraction metrics, see \cite{Bullo_2022}. A particularly important observation is that if a system has a fixed point, then (exponential) contraction gives asymptotic stability. Moreover, it follows that the fixed point is unique. Exponential contraction with rate $\gamma$ also implies that the worst-case error decreases proportionally to $\gamma^k$. This is often called \textit{linear convergence} in the optimization literature. 

In what follows, we develop an approach to check whether the interconnection of systems of the form \eqref{eq:sigma} coupled by $u=My$ is contractive for a specific class of oracles.

\section{Incremental Dissipativity Approach}
\subsection{Incremental Dissipativity}

Recall that the oracles $\varphi_i$ are related to the local cost functions $f_i$. Since the design of optimization algorithms should not be limited to a single problem, we aim to show convergence for relevant classes of cost functions and corresponding oracles. Some relevant properties are the following:

\begin{definition}\label{def:convex}
    Let $\mu, K >0$. A function $f:\mathbb{R}^n\to \mathbb{R}$ is called $\mu$-strongly convex if, for all $x_1,x_2\in\mathbb{R}^n$ and $\alpha\in(0,1)$, we have
    \begin{align*}
        f(\alpha x_1 + (1-\alpha) x_2) & \leq \alpha f(x_1) + (1-\alpha) f(x_2)\\
        & \quad -\alpha(1-\alpha)\frac{\mu}{2} \norm{x_1-x_2}_2^2.
    \end{align*}
    Moreover, a function $\varphi:\mathbb{R}^n\to \mathbb{R}^n$ is called $\mu$-strongly monotone if
    \begin{equation*}
        (\varphi(x_1) - \varphi(x_2))^\top (x_1-x_2) \geq \mu \norm{x_1-x_2}_2^2
    \end{equation*}
    and $K$-Lipschitz if
    \begin{equation*}
        \norm{\varphi(x_1)-\varphi(x_2)}_2^2 \leq K^2 \norm{x_1 - x_2}_2^2
    \end{equation*}
    for all $x_1, x_2 \in \mathbb{R}^n$.
\end{definition}

These properties, and others derived from monotone operator theory \cite{Bauschke_Combettes_2017}, can be expressed as incremental bounds of the form
\begin{equation}\label{eq:incr-bound}
    s_{\varphi_i}(\Delta z_i, \Delta w_i):=
    \begin{bmatrix}
        \Delta z_i \\ \Delta w_i
    \end{bmatrix}^\top S_{\varphi_i} \begin{bmatrix}
        \Delta z_i \\ \Delta w_i
    \end{bmatrix} \leq 0,
\end{equation}
which holds for all $\Delta z_i$ and corresponding $\Delta w_i$, with an appropriate matrix $S_{\varphi_i}$. This is called a \textit{sector bound} in terms of the increments $\Delta w_i$ and $\Delta z_i$.

In particular, it is well known that the gradient of a $\mu$-strongly convex differentiable function is $\mu$-strongly monotone. Then, for gradient-based methods, where $\varphi_i = \nabla f_i$, a $\mu$-strongly convex cost function $f_i$ implies a sector bound given by 
\begin{equation}\label{eq:sector-bound-monotone}
    S_{\textrm{mon}} = \begin{bmatrix}
        2\mu I & -I\\
        -I & 0
    \end{bmatrix},
\end{equation}
which follows immediately from Definition~\ref{def:convex}. A $K$-Lipschitz oracle satisfies a similar quadratic sector bound.
Several further examples with different oracles are given in \cite{Lessard_2022}. If an oracle satisfies several sector bounds, it also satisfies any conic combination of them, allowing us to exploit more than one known property simultaneously.

Given sector bounds on the oracles $\varphi_i$, our goal is to prove that the interconnection of $\Sigma_i$, $i=1,2,\dots,N$, is contractive under any nonlinear feedback at each $\Sigma_i$ that satisfies the corresponding sector bound. 
To tackle this problem, the notion of contraction needs to be extended to systems with inputs and outputs. This leads to the notion of incremental dissipativity \cite{Sepulchre_Chaffey_Forni_2022}, analogously to how dissipativity theory \cite{Willems_1972a, Willems_1972b} generalises Lyapunov theory to open systems.
\begin{definition}
    Consider a discrete-time dynamical system $\Sigma$ with state $x\in\mathbb{R}^n$, input $u\in\mathbb{R}^m$, and output $y\in\mathbb{R}^p$, and a supply rate $s:\mathbb{R}^{p}\times\mathbb{R}^m\to \mathbb{R}$. The system $\Sigma$ is said to be
    incrementally dissipative with respect to $s$, with dissipation rate $\gamma\in(0,1]$, if there exists a positive definite function $V:\mathbb{R}^n\to\mathbb{R}$ such that
    \begin{equation*}
            V(\Delta x^+) - \gamma V(\Delta x) \leq s(\Delta y, \Delta u)
    \end{equation*}
    for all incremental trajectories $\Delta u, \Delta x, \Delta y$.
\end{definition}

In \cite{Willems_1972b} it is shown that an LTI system is dissipative with respect to a quadratic supply rate $s$ if and only if it is dissipative with respect to $s$ with a quadratic storage function $V(x)=x^\top P x$, where $P\succ 0$, and this result carries over to incremental dissipativity \cite{onlineopt}. Since we consider only $\Sigma_i$ which are LTI systems and we work with quadratic supply rates in the sequel, we restrict our attention to positive definite quadratic storage functions.

\subsection{Analysis of Distributed Algorithms}

We aim to show that the interconnection described in Section~\ref{sec:network-system-formulation} is contractive. In the minimal case of $N=1$, we would seek incremental dissipativity of the LTI system $\Sigma_1$ with respect to the supply rate given by a sector bound $s_{\varphi_1}$ of the form \eqref{eq:incr-bound}, following the ideas of \cite{Lessard_2022}. Then, the nonpositive supply rate would imply that the closed loop system is at least nonexpansive.

In the general setting, we show contraction by, on the one hand, verifying that each $\Sigma_i$ is incrementally dissipative with respect to a local supply rate that quantifies both the oracle feedback and the network interconnection, and, on the other hand, checking that the sum of these supply rates over the network is negative.
In particular, we are interested in local supply rates of the form
\begin{align}\label{eq:local-supply-rate}
    &s_i(\Delta z_i, \Delta w_i, \Delta y_i, \Delta u_i) \nonumber \\
    &= \alpha_i s_{\varphi_i} (\Delta z_i, \Delta w_i) + s_{\textup{ext},i}(\Delta y_i, \Delta u_i),
\end{align}
where $\alpha_i>0$, $s_{\varphi_i}$ is as in \eqref{eq:incr-bound}, and
\begin{equation}\label{eq:interconnection-supply-rate}
    s_{\textup{ext},i}(\Delta y_i, \Delta u_i) = \begin{bmatrix}
            \Delta y_i \\ \Delta u_i
        \end{bmatrix}^\top \begin{bmatrix}
            Q_i & S_i\\
            S_i^\top & R_i
        \end{bmatrix}\begin{bmatrix}
            \Delta y_i \\ \Delta u_i
        \end{bmatrix}
\end{equation}
with matrices $Q_i, S_i, R_i$ that are parameters used to capture interaction over the network. We require $Q_i^\top = Q_i$ and $R_i^\top = R_i$. For the sake of conciseness, we assume here that only one sector bound is used, but our arguments can be extended by replacing the first term in \eqref{eq:local-supply-rate} with a free conic combination of several sector bounds.
Note that \eqref{eq:local-supply-rate} is a quadratic supply rate. Consequently, we can test for incremental dissipativity through linear matrix inequalities (LMIs), as stated in Lemma~\ref{lem:network-diss-LMI}.

\begin{lemma}\label{lem:network-diss-LMI}
    Given $Q_i, S_i, R_i$ and $\alpha_i>0$, the system \eqref{eq:sigma-tilde-LTI} is incrementally dissipative with respect to $s_i$ as in \eqref{eq:local-supply-rate} with dissipation rate $\gamma\in(0,1]$ and storage function $V_i:\Delta x_i\mapsto \Delta x_i^\top P_i \Delta x_i$, $P_i\succ 0$, if and only if

    \begin{align}
        &\setlength\arraycolsep{3.5pt} \begin{bmatrix}
            A_i & B_i & G_i\\
            I & 0 & 0
        \end{bmatrix}^\top \begin{bmatrix}
            -P_i & 0\\
            0 & \gamma P_i
        \end{bmatrix}\begin{bmatrix}
            A_i & B_i & G_i\\
            I & 0 & 0
        \end{bmatrix} \nonumber\\
        & \setlength\arraycolsep{3.5pt} + \alpha_i \begin{bmatrix}
            C_i^{\textrm{opt}} & D_i^{\textrm{opt}} & H_i^{\textrm{opt}}\\
            0 & 0 & I
        \end{bmatrix}^\top S_{\varphi_i} \begin{bmatrix}
            C_i^{\textrm{opt}} & D_i^{\textrm{opt}} & H_i^{\textrm{opt}}\\
            0 & 0 & I
        \end{bmatrix} \label{eq:local-dissipation-lmi}\\
        & \setlength\arraycolsep{3.5pt} +  \begin{bmatrix}
            C_i^{\textrm{con}} & D_i^{\textrm{con}} & H_i^{\textrm{con}}\\
            0 & I & 0
        \end{bmatrix}^\top \begin{bmatrix}
            Q_i & S_i\\
            S_i^\top & R_i
        \end{bmatrix}
        \begin{bmatrix}
            C_i^{\textrm{con}} & D_i^{\textrm{con}} & H_i^{\textrm{con}}\\
            0 & I & 0
        \end{bmatrix}\succeq 0. \nonumber
    \end{align}
\end{lemma}
\vspace{1em}
\begin{proof}
    Sufficiency follows from pre- and postmultiplying \eqref{eq:local-dissipation-lmi} with $[\Delta x_i^\top \hspace{0.5em} \Delta u_i^\top \hspace{0.5em} \Delta w_i^\top]$ and $[\Delta x_i^\top \hspace{0.5em} \Delta u_i^\top \hspace{0.5em} \Delta w_i^\top]^\top$, respectively. To prove necessity, note that the dissipation inequality is required to hold for all initial values of $\Delta x_i$ and all inputs $\Delta u_i$, $\Delta w_i$.
\end{proof}

Lemma~\ref{lem:network-diss-LMI} provides an LMI feasibility condition in decision variables $P_i, Q_i, S_i, R_i$, and $\alpha_i$ for fixed $\gamma$. The smallest feasible $\gamma$ can be found by bisection. In addition, for many algorithms, the matrices in \eqref{eq:sigma-tilde-LTI} have a Kronecker product structure that makes the number of decision variables of the LMIs independent of $n$; see \cite{Lessard_2022} for details.

Suppose that the oracle $\varphi_i$ satisfies the incremental sector bound $s_{\varphi_i}(\Delta z_i, \Delta w_i)$ as in \eqref{eq:incr-bound}. Then, if the local LTI system $\Sigma_i$ is incrementally dissipative with respect to a supply rate of the form \eqref{eq:local-supply-rate}, we can conclude that \eqref{eq:sigma-tilde-LTI} interconnected with the oracle $\eqref{eq:sigma-tilde-oracle}$ is incrementally dissipative with respect to $s_{\textup{ext},i}(\Delta y_i, \Delta u_i)$.
We can now use the ideas of \cite{Martinelli_QSR} to relate local parametric quadratic storage functions to the network structure as follows.

\begin{theorem}\label{thm:general-static-interconnection}
    Consider systems of the form \eqref{eq:sigma} for $i=1,2,\dots, N$, coupled via $u=My$. Assume that
    \begin{enumerate}
        \item each oracle $\varphi_i$ satisfies an incremental bound $s_{\varphi_i}(\Delta z_i, \Delta w_i) \leq 0$, where $s_{\varphi_i}$ is as in \eqref{eq:incr-bound} and
        \item there exist parameters $\alpha_i, Q_i, S_i, R_i$ as in \eqref{eq:local-supply-rate}, and $\gamma\in(0,1)$ such that each $\Sigma_i$ is incrementally dissipative with respect to $s_i$ with dissipation rate $\gamma$.
    \end{enumerate}
    Then, the interconnection is exponentially contractive with rate $\gamma$ if
    \begin{equation}\label{eq:general-interconnection-lmi}
        \begin{bmatrix}
            I \\ M
        \end{bmatrix}^\top \begin{bmatrix}
            Q & S\\
            S^\top & R
        \end{bmatrix}\begin{bmatrix}
            I \\ M
        \end{bmatrix} \preceq 0,
    \end{equation}
    $Q = \diag(Q_1, Q_2, \dots, Q_N)$, $S = \diag(S_1, S_2, \dots, S_N)$, and $R=\diag(R_1, R_2, \dots, R_N)$.
\end{theorem}
\begin{proof}
    Assume, without loss of generality, that the storage function of each $\Sigma_i$ is quadratic, and denote it by $V_i$. Define a global storage function $V$ by
    \begin{equation}
        V(\Delta x) = \sum_{i=1}^N V_i(\Delta x_i)
    \end{equation}
    and let 
    \begin{equation}
        s_\varphi(\Delta z, \Delta w)=\sum_{i=1}^N \alpha_i s_{\varphi_i} (\Delta z_i, \Delta w_i).
    \end{equation}
    Then, summing over the index $i$ gives
    \begin{align}
        &V(\Delta x^+)-\gamma V(\Delta x)\\
        & \leq  s_\varphi(\Delta z, \Delta w) + \begin{bmatrix}
            \Delta y \\ \Delta u
        \end{bmatrix} ^\top \begin{bmatrix}
            Q & S\\
            S^\top & R
        \end{bmatrix}
        \begin{bmatrix}
            \Delta y \\ \Delta u
        \end{bmatrix}\\
        &= s_\varphi(\Delta z, \Delta w) + \Delta y^\top \begin{bmatrix}
            I \\ M
        \end{bmatrix} ^\top \begin{bmatrix}
            Q & S\\
            S^\top & R
        \end{bmatrix}
        \begin{bmatrix}
            I \\ M
        \end{bmatrix}\Delta y.
    \end{align}
    By \eqref{eq:general-interconnection-lmi} and the sector bounds \eqref{eq:incr-bound}, we obtain exponential contraction with rate $\gamma$.
\end{proof}

Observe that Theorem~\ref{thm:general-static-interconnection} relates incremental dissipativity properties of the individual systems in the network to the interconnection structure. Since the LMI \eqref{eq:general-interconnection-lmi} captures all relevant information related to the interconnection structure, this test can show convergence to a fixed point with any network topology. That said, some assumptions, such as connectedness, may be necessary to ensure that the fixed point indeed solves the optimization problem.

Assuming that a candidate for a distributed optimization algorithm has an appropriate fixed point, the above results provide a step-by-step recipe to verify its convergence:
\begin{enumerate}
    \item Rewrite the algorithm in the form presented in Section~\ref{sec:network-system-formulation}.
    \item Use the properties of the local cost functions $f_i$ and the corresponding oracles used in the algorithm to derive incremental bounds of the form \eqref{eq:incr-bound} for each oracle.
    \item Check for the simultaneous feasibility of \eqref{eq:local-dissipation-lmi} for $i=1,2,\dots, N$ and \eqref{eq:general-interconnection-lmi} with a candidate contraction rate $\gamma\in (0,1)$, using semidefinite programming. Perform bisection to find the smallest admissible $\gamma$.
\end{enumerate}

If this procedure fails with all $\gamma\in(0,1)$, we may still be able to verify contraction with rate $\gamma=1$, or, in optimization terms, sublinear convergence. This is formalized in the following theorem.

\begin{theorem}\label{prop:general-interconnection-weak}
    Consider systems of the form \eqref{eq:sigma} and coupling $u=My$. Assume that
    \begin{enumerate}
        \item each $(C_i^\textrm{opt}, A_i)$ pair is detectable and
        \item each $\varphi_i$ satisfies a sector bound $s_{\varphi_i}(\Delta z_i, \Delta w_i) < 0$ whenever $(\Delta z_i, \Delta w_i) \neq (0,0)$, where $s_{\varphi_i}$ is defined as in \eqref{eq:incr-bound}.
    \end{enumerate}

    If each $\Sigma_i$ is incrementally dissipative with respect to $s_i$ as in \eqref{eq:local-supply-rate} and, in the notation of Theorem~\ref{thm:general-static-interconnection},
    \begin{equation}\label{eq:general-interconnection-lmi-strict}
        \begin{bmatrix}
            I \\ M
        \end{bmatrix}^\top \begin{bmatrix}
            Q & S\\
            S^\top & R
        \end{bmatrix}\begin{bmatrix}
            I \\ M
        \end{bmatrix} \prec 0
    \end{equation}
    holds, then the interconnection is nonexpansive and
    \begin{equation*}
        \lim_{k\to \infty} \Delta x(k)=0.
    \end{equation*}
\end{theorem}
    \vspace{1em}
\begin{proof}
    Define $V$ as in the proof of Theorem~\ref{thm:general-static-interconnection}. Then $V(\Delta x^+)-V(\Delta x) \leq 0$ if $\Delta u, \Delta w, \Delta y, \Delta z$ are all zero and negative otherwise.

    Let $\beta > 0$ and $\Theta_\beta = \{\Delta x \mid V(\Delta x) \leq \beta\}$. By the above, $\Theta_\beta$ is forward-invariant. Since $V$ is positive definite and radially unbounded, $\Theta_\beta$ is also compact. Define $\Omega_\beta=\{\Delta x \in \Theta_\beta \mid V(\Delta x^+) - V(\Delta x)=0\}$ and let $E_\beta$ be the largest invariant set contained in $\Omega_\beta$.

    Then, by LaSalle's Invariance Principle, all incremental trajectories starting in $\Theta_\beta$ converge to $E_\beta$. Note that all inputs and outputs must be identically zero on $E_\beta$. Hence, the coupling and the nonlinearities disappear, and the detectability assumption allows us to conclude that $E_\beta = \{0\}$.

    Finally, we observe that $\beta$ can be arbitrarily large to prove the claim for any initial condition.
\end{proof}

Sublinear convergence can be proven using Theorem~\ref{prop:general-interconnection-weak} by following the procedure above with slight modifications. That is, we set $\gamma = 1$, verify that each pair $(C_i^\textrm{opt}, A_i)$ is detectable, and derive a strict sector bound for each oracle. Finally, we use the LMI \eqref{eq:general-interconnection-lmi-strict} in place of \eqref{eq:general-interconnection-lmi}.

Strict sector bounds can often be found by varying the constants in non-strict ones. For example, a $\Tilde{\mu}$-strongly convex oracle satisfies a strict sector bound given by \eqref{eq:sector-bound-monotone} if $\mu < \Tilde{\mu}$.

Observe that this analysis approach does not assume that the local algorithms are identical. In particular, they can be chosen or tuned to best suit their corresponding cost functions $f_i$ and their positions in the network.

\section{Comparison with Classical Analysis}

In this section, we compare the conditions for convergence obtained with our method with those from the literature, in the case of distributed gradient descent, through simulations.

Consider distributed gradient descent from Example~\ref{ex:DGD}, for simplicity with $n=1$.
Assume that each local cost function $f_i$ is $\mu$-strongly convex and has $K$-Lipschitz gradients. Then, by \cite{Yuan_Ling_Yin_2016}, the algorithm converges if, in our notation,
\begin{equation}\label{eq:eta-bound}
    \rho < \frac{1}{d_{\textrm{max}}} \quad \textrm{and}  \quad 
    \eta < \frac{2-\rho \lambda_{\textrm{max}}(L)}{K},
\end{equation}
where $d_\textrm{max}$ is the maximum degree of the graph and $\lambda_{\textrm{max}}(L)$ is the largest eigenvalue of $L$. However, the convergence of distributed gradient descent with constant step size is only approximate\footnote{Exact convergence can be achieved with diminishing step sizes.} in the sense that estimates $x_i$ converge to points where the cost error is in $\mathcal{O}(\frac{\eta}{1-\sigma})$, where $\sigma$ is the second largest magnitude of eigenvalues of $I-\rho L$.
Consequently, linear convergence does not hold globally \cite{Yuan_Ling_Yin_2016}. Hence, we aim to use Theorem~\ref{prop:general-interconnection-weak}, noting that detectability is immediate in \eqref{eq:dgd-lti}.

Using the fact that each oracle is $\mu$-strongly monotone and $K$-Lipschitz, we derive incremental sector bounds defined by
\begin{equation}
    S_{\varphi_i} = \begin{bmatrix}
        2K\mu & -K-\mu\\
        -K-\mu & 2
    \end{bmatrix}.
\end{equation}
Note that Theorem~\ref{prop:general-interconnection-weak} requires a strict sector bound, which can, without loss of generality, be imposed by considering an infinitesimal increase in $K$ and decrease in $\mu$. 

Substituting this sector bound and the dynamics \eqref{eq:dgd-lti} into Lemma~\ref{lem:network-diss-LMI}, the LMI
\begin{equation}\label{eq:dgd-local-LMI}
    \begin{bmatrix}
        2 \alpha_i K \mu +Q_i & \rho P_i + S_i & \eta P_i - \alpha_i(K+\mu)\\
        \rho P_i + S_i^\top & -\rho^2 P_i + R_i & -\rho\eta P_i\\
        \eta P_i - \alpha_i(K+\mu) & -\rho \eta P_i & -\eta^2 P_i + 2\alpha_i
    \end{bmatrix}\succeq 0
\end{equation}
has to be feasible with decision variables $P_i, \alpha_i >0$ and $Q_i, S_i, R_i\in\mathbb{R}$ for $i=1,2,\dots, N$, simultaneously with
\begin{equation}\label{eq:Laplacian-global-LMI}
    Q + L S^\top + S L + L R L \prec 0,
\end{equation}
where $Q, S, R$ are defined as in Theorem~\ref{thm:general-static-interconnection}, to conclude contraction by Theorem~\ref{prop:general-interconnection-weak}.

Consider the communication graph depicted in Figure~\ref{fig:graph} and let $\mu=0.05$, $K=1$. Figure~\ref{fig:grid-search-basic} shows the results of simultaneous feasibility tests of \eqref{eq:dgd-local-LMI} and \eqref{eq:Laplacian-global-LMI} over a $\rho,\eta$ grid with a lower bound of $0.001$ and increments of $0.05$. This has been obtained with CVX~\cite{cvx, gb08}.

Our approach is able to identify all parameter choices that are deemed valid by the traditional analysis. We also find an additional region in the $\rho,\eta$ space where contraction is certified. These results were verified for a range of variations of $\mu$, $K$, and the structure of the graph.

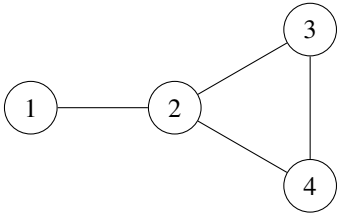
\begin{figure}
    \centering
    \vspace{1.5mm}
    \begin{tikzpicture}[every node/.style={circle, draw, minimum size=0.7cm}]
  \node (2) at (180:1cm) {2};
  
  \node (1) [left=1.4cm of 2] {1};
  \node (3) [right=1.4cm of 2] {3};
  \node (4) [below=0.5cm of 2] {4};

  \draw (2) -- (3) -- (4) -- (2);
  \draw (2) -- (1);
\end{tikzpicture}
    \caption{Communication graph used in the example.}
    \label{fig:graph}
\end{figure}

\begin{figure}
    \centering
    \begin{subfigure}{0.28\linewidth}
        \centering
        \includegraphics[height=4.75cm, trim={6.5cm 0 6cm 1cm}, clip]{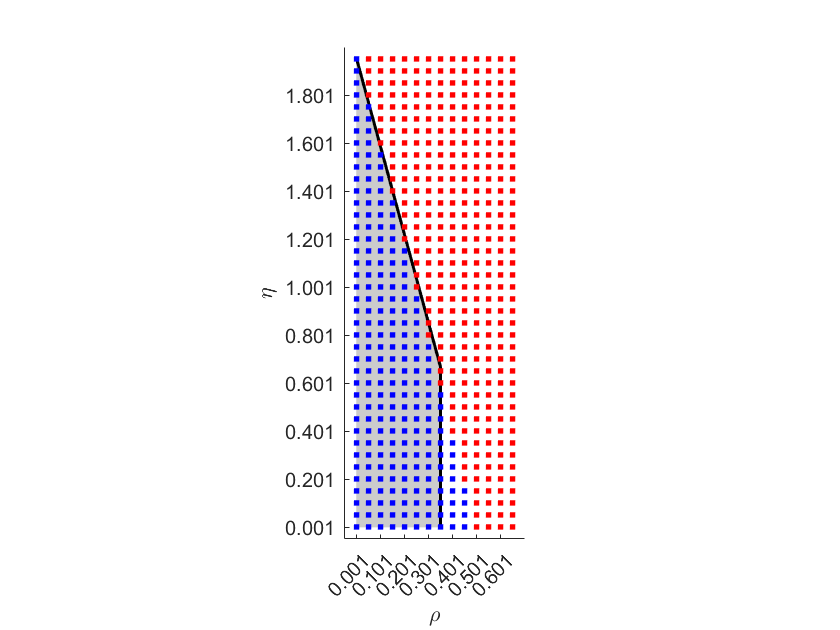}
        \caption{All agents.}
        \label{fig:grid-search-basic}
    \end{subfigure}\hfill\begin{subfigure}{0.4\linewidth}
        \centering
        \includegraphics[height=4.75cm, trim={4.8cm 0 4.8cm 1cm}, clip]{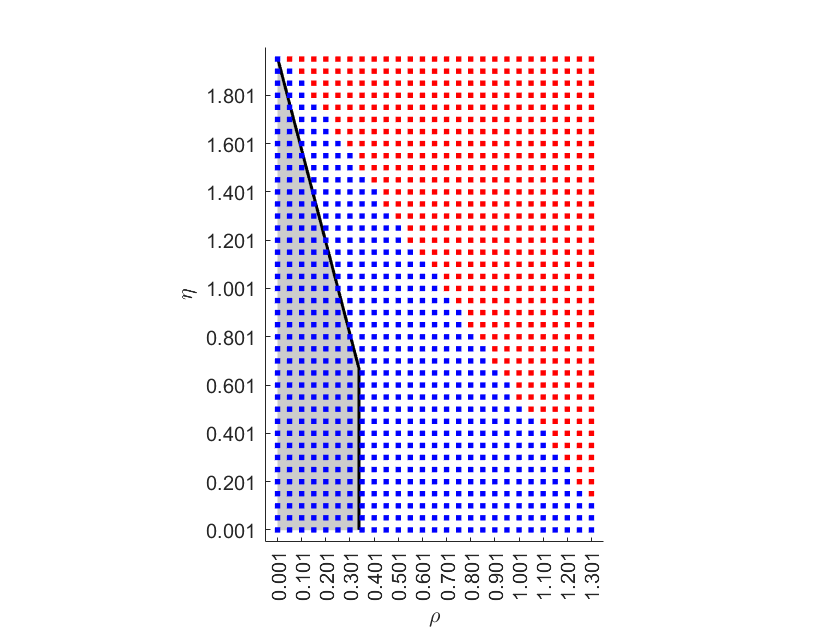}
        \caption{Agent 1.}
        \label{fig:agent-1}
    \end{subfigure}\hfill\begin{subfigure}{0.28\linewidth}
        \centering
        \includegraphics[height=4.75cm, trim={6.5cm 0 6cm 1cm}, clip]{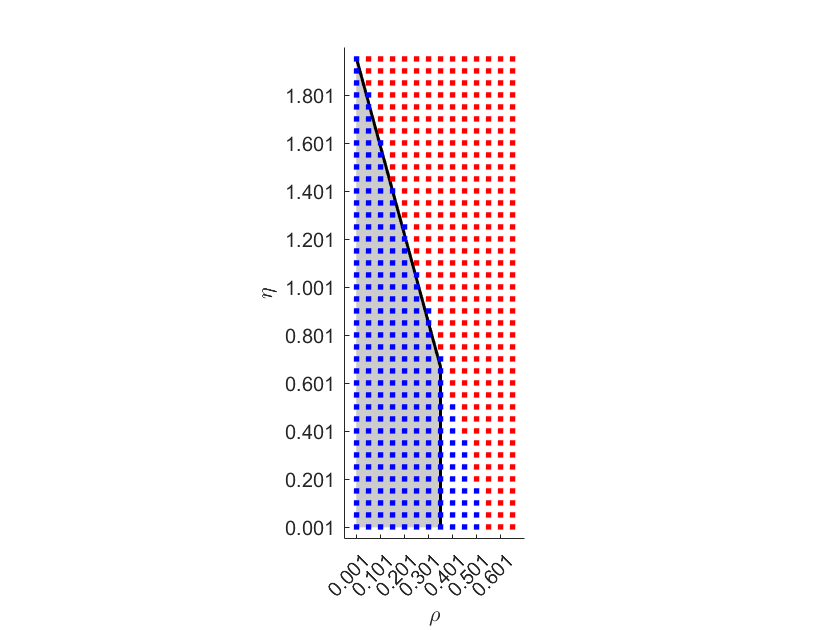}
        \caption{Agent 2.}
        \label{fig:agent-2}
    \end{subfigure}
    \caption{Grid search over $\rho$ and $\eta$, with $\mu=0.05$ and $K = 1$, using the communication graph in Fig.~\ref{fig:graph}. Blue squares show where the LMIs \eqref{eq:dgd-local-LMI} and \eqref{eq:Laplacian-global-LMI} are feasible, while the gray area is where \eqref{eq:eta-bound} hold. In (b) and (c), a single agent is tuned with all other step sizes fixed at $(\rho_i, \eta_i)=(0.35, 0.025)$.}
\end{figure}

\begin{figure}
    \centering
    \vspace{2mm}
    \includegraphics[width=0.85\linewidth]{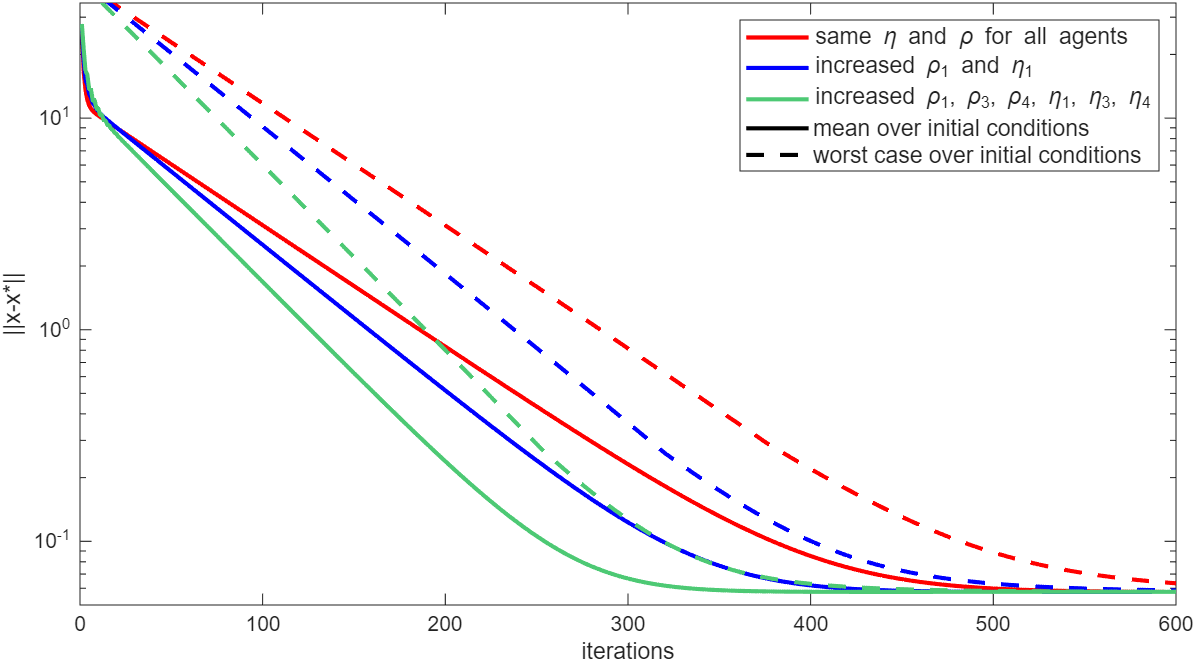}
    \caption{Comparison of logarithmic error over 1000 uniformly random initial conditions from $(-25, 25)^4$. We compare a baseline using $(\rho_i, \eta_i)=(0.35, 0.025)$ at all agents (red); setting $(\rho_1, \eta_1)=(1.05, 0.075)$, while keeping the rest unchanged (blue); and where in addition $(\rho_i, \eta_i)=(0.525, 0.0375)$ for $i=3,4$ (green). }
    \label{fig:error-curve}
\end{figure}
 
A major advantage of our analysis approach is that it allows us to verify contraction of heterogeneous algorithms, where each agent uses different parameters. This is illustrated in Figures~\ref{fig:agent-1} and \ref{fig:agent-2}, which show the step sizes of agents 1 and 2, respectively, which lead to contraction with the other step sizes fixed at $(\rho_i, \eta_i)=(0.35, 0.025)$. 
Observe that for agent 1, which is a leaf in the network, we obtain a significantly larger set of valid parameters. This highlights that the dissipativity approach is able to exploit the network structure and verify or accelerate convergence through heterogeneous step sizes.
In Figure~\ref{fig:error-curve} we consider 
\begin{align*}
    &f_1(x_1)=0.125(x_1-1)^2,  &&f_2(x_2)=0.4(x_2-3)^2, \\
    &f_3(x_3)=0.475(x_3+0.5)^2, &&f_4(x_4)=0.06(x_4-4)^2,
\end{align*}
and compare the error curves for three different sets of parameters. To be precise, contraction is verified for all of these choices using the LMIs~\eqref{eq:dgd-local-LMI} and \eqref{eq:Laplacian-global-LMI}. Observe that heterogeneous tuning leads to faster convergence.

We can similarly exploit differences in the sector bounds across the network in cases where the local cost functions are known to have different properties. This avoids the conservatism of using worst-case parameters at all agents.

\section{Conclusion}
We have introduced a dissipativity-based framework for the analysis of distributed optimization algorithms. This framework yields a systematic pipeline for proving convergence of possibly heterogeneous algorithms with arbitrary network structure, through the notion of contraction, with conditions expressed as linear matrix inequalities.

Our simulations suggest that this approach can verify the convergence of distributed gradient descent for all parameter choices that satisfy the classical sufficient conditions and provides additional freedom in tuning the algorithm, by exploiting the network structure.

Promising research directions include utilizing this framework and the theory on dissipativity-based synthesis for the design of new algorithms, as well as applying our methods to optimization-in-the-loop control of distributed systems.

\bibliography{bibliography}
\bibliographystyle{IEEEtran}

\end{document}